\documentclass[12pt, reqno]{amsart}
\usepackage{amsmath, amsthm, amscd, amsfonts, amssymb, graphicx, color}
\usepackage[bookmarksnumbered, colorlinks, plainpages]{hyperref}
\hypersetup{colorlinks=true,linkcolor=red, anchorcolor=green,
citecolor=cyan, urlcolor=red, filecolor=magenta, pdftoolbar=true}

\newtheorem{theorem}{Theorem}[section]
\newtheorem{lemma}[theorem]{Lemma}

\newtheorem{corollary}[theorem]{Corollary}
\theoremstyle{definition}

\theoremstyle{remark}
\newtheorem{remark}[theorem]{Remark}
\numberwithin{equation}{section}

\begin{document}

\setcounter{page}{1}

\title[$A$-numerical radius inequalities]
{$A$-numerical radius inequalities for semi-Hilbertian space operators}

\author[A. Zamani]{Ali Zamani}

\address{Department of Mathematics, Farhangian University, Tehran, Iran}
\email{zamani.ali85@yahoo.com}



\begin{abstract}
Let $A$ be a positive bounded operator on a Hilbert space 
$\big(\mathcal{H}, \langle \cdot, \cdot\rangle \big)$.
The semi-inner product ${\langle x, y\rangle}_A := \langle Ax, y\rangle$, $x, y\in\mathcal{H}$
induces a semi-norm ${\|\cdot\|}_A$ on $\mathcal{H}$.
Let ${\|T\|}_A$ and $w_A(T)$ denote the $A$-operator semi-norm
and the $A$-numerical radius of an operator $T$ in semi-Hilbertian space
$\big(\mathcal{H}, {\|\cdot\|}_A\big)$, respectively.
In this paper, we prove the following characterization of $w_A(T)$
\begin{align*}
w_A(T) = \displaystyle{\sup_{\alpha^2 + \beta^2 = 1}}
{\left\|\alpha \frac{T + T^{\sharp_A}}{2} +
\beta \frac{T - T^{\sharp_A}}{2i}\right\|}_A,
\end{align*}
where $T^{\sharp_A}$ is a distinguished $A$-adjoint operator of $T$.
We then apply it to find upper and lower bounds for $w_A(T)$.
In particular, we show that
\begin{align*}
\frac{1}{2}{\|T\|}_A \leq \max\Big\{\sqrt{1 - {|\cos|}^2_AT}, \frac{\sqrt{2}}{2}\Big\}w_A(T)\leq w_A(T),
\end{align*}
where ${|\cos|}_AT$ denotes the $A$-cosine of angle of $T$.
Some upper bounds for the $A$-numerical radius of commutators, anticommutators,
and products of semi-Hilbertian space operators are also given.
\end{abstract} \maketitle
\section{Introduction and preliminaries}
Let $\mathbb{B}(\mathcal{H})$ denote the $C^{\ast}$-algebra of all bounded
linear operators on a complex Hilbert space $\big(\mathcal{H}, \langle \cdot, \cdot\rangle \big)$ and the corresponding norm
$\|\cdot\| $. Let the symbol $I$ stand for the identity operator on $\mathcal{H}$.
If $T\in\mathbb{B}(\mathcal{H})$, then we denote by $\mathcal{R}(T)$
the range of $T$, and by $\overline{\mathcal{R}(T)}$ the norm closure of $\mathcal{R}(T)$.
Throughout this paper, we assume that $A\in\mathbb{B}(\mathcal{H})$ is a positive operator
and that $P$ is the orthogonal projection onto $\overline{\mathcal{R}(A)}$.
Recall that $A$ is called positive, denoted by $A\geq 0$,
if $\langle Ax, x\rangle\geq 0$ for all $x\in\mathcal{H}$.
Such an operator $A$ induces a positive semidefinite sesquilinear
form ${\langle \cdot, \cdot\rangle}_A: \,\mathcal{H}\times \mathcal{H} \rightarrow \mathbb{C}$
defined by ${\langle x, y\rangle}_A = \langle Ax, y\rangle$, \,$x, y\in\mathcal{H}$.
Denote by ${\|\cdot\|}_A$ the seminorm induced by ${\langle \cdot, \cdot\rangle}_A$, that is,
${\|x\|}_A = \sqrt{{\langle x, x\rangle}_A}$
for every $x\in\mathcal{H}$.
It can be easily seen that ${\|\cdot\|}_A$ is a norm if and only if $A$ is an injective operator,
and  that $(\mathcal{H}, {\|\cdot\|}_A)$ is a complete space if and only if $\mathcal{R}(A)$ is closed in $\mathcal{H}$.
The semi-inner product ${\langle \cdot, \cdot\rangle}_A$ induces a semi-norm on a certain subspace of $\mathbb{B}(\mathcal{H})$.
Namely, given $T\in\mathbb{B}(\mathcal{H})$, if there exists $c > 0$ such that ${\|Tx\|}_A \leq c{\|x\|}_A$
for all $x \in \overline{\mathcal{R}(A)}$, then it holds that
\begin{align*}
{\|T\|}_A := \displaystyle{\sup_{x \in \overline{\mathcal{R}(A)}, x\neq 0}} \frac{{\|Tx\|}_A}{{\|x\|}_A}
= \inf\Big\{c>0: \,\, {\|Tx\|}_A \leq c{\|x\|}_A, x\in \mathcal{H}\Big\}< \infty.
\end{align*}
We set $\mathbb{B}^{A}(\mathcal{H}) :=\Big\{T\in \mathbb{B}(\mathcal{H}): \, \, {\|T\|}_A < \infty\Big\}$.
It can be seen that $\mathbb{B}^{A}(\mathcal{H})$ is not generally a subalgebra of $\mathbb{B}(\mathcal{H})$
and ${\|T\|}_A = 0$ if and only if $ATA = 0$.
In addition, for $T\in \mathbb{B}^{A}(\mathcal{H})$, we have
\begin{align*}
{\|T\|}_A = \sup\Big\{|{\langle Tx, y\rangle}_A|:\, \, x, y\in \overline{\mathcal{R}(A)}, {\|x\|}_A = {\|y\|}_A = 1\Big\}.
\end{align*}
An operator $T$ is called $A$-positive if $AT\geq 0$.
Note that if $T$ is $A$-positive, then
\begin{align*}
{\|T\|}_A = \sup\Big\{{\langle Tx, x\rangle}_A:\, \, x\in \mathcal{H}, {\|x\|}_A = 1\Big\}.
\end{align*}
For $T\in \mathbb{B}(\mathcal{H})$, an operator $R\in \mathbb{B}(\mathcal{H})$
is called an $A$-adjoint of $T$ if for every $x, y\in \mathcal{H}$,
we have ${\langle Tx, y\rangle}_A = {\langle x, Ry\rangle}_A$, i.e., $AR = T^*A$.
The existence of an $A$-adjoint operator is not guaranteed. In fact, an operator $T\in \mathbb{B}(\mathcal{H})$ may admit none, one or
many $A$-adjoints. The set of all operators which admit $A$-adjoints is denoted by $\mathbb{B}_{A}(\mathcal{H})$.
Note that $\mathbb{B}_{A}(\mathcal{H})$ is a subalgebra of $\mathbb{B}(\mathcal{H})$ which is neither closed nor dense
in $\mathbb{B}(\mathcal{H})$. Moreover, the following inclusions
$\mathbb{B}_{A}(\mathcal{H}) \subseteq \mathbb{B}^{A}(\mathcal{H}) \subseteq \mathbb{B}(\mathcal{H})$ hold
with equality if $A$ is injective and has a closed range.

If $T\in\mathbb{B}_{A}(\mathcal{H})$, the reduced solution of the equation $AX = T^*A$ is a distinguished $A$-adjoint
operator of $T$, which is denoted by $T^{\sharp_A}$; see \cite{M.K.X}.
Note that, $T^{\sharp_A} = A^{\dag}T^*A$ in which $A^{\dag}$ is the Moore--Penrose inverse of $A$.
It is useful that if $T\in\mathbb{B}_{A}(\mathcal{H})$, then $AT^{\sharp_A} = T^*A$.
An operator $T\in\mathbb{B}(\mathcal{H})$ is said to be $A$-selfadjoint if $AT$ is selfadjoint,
i.e., $AT = T^*A$. Observe that if $T$ is $A$-selfadjoint, then $T\in\mathbb{B}_{A}(\mathcal{H})$.
However it does not hold, in general, that $T = T^{\sharp_A}$.
For example, consider the operators
$A = \begin{bmatrix}
1 & 1 \\
1 & 1
\end{bmatrix}$
and
$T = \begin{bmatrix}
2 & 2 \\
0 & 0
\end{bmatrix}$.
Then simple computations show that $T$ is $A$-selfadjoint and
$T^{\sharp_A} = \begin{bmatrix}
1 & 1 \\
1 & 1
\end{bmatrix} \neq T$.
More precisely, if $T\in\mathbb{B}_{A}(\mathcal{H})$, then $T = T^{\sharp_A}$ if and only if
$T$ is $A$-selfadjoint and $\mathcal{R}(T) \subseteq \overline{\mathcal{R}(A)}$.
Notice that if $T\in\mathbb{B}_{A}(\mathcal{H})$, then $T^{\sharp_A}\in\mathbb{B}_{A}(\mathcal{H})$,
$(T^{\sharp_A})^{\sharp_A} = PTP$ and $\big((T^{\sharp_A})^{\sharp_A}\big)^{\sharp_A} = T^{\sharp_A}$.
In addition, $T^{\sharp_A}T$, $TT^{\sharp_A}$ are $A$-selfadjoint and $A$-positive and so we have
\begin{align*}
{\|T^{\sharp_A}T\|}_A = {\|TT^{\sharp_A}\|}_A = {\|T\|}^2_A = {\|T^{\sharp_A}\|}^2_A.
\end{align*}
Furthermore, if $T, S\in\mathbb{B}_{A}(\mathcal{H})$, then $(TS)^{\sharp_A} = S^{\sharp_A}T^{\sharp_A}$,
${\|TS\|}_A \leq {\|T\|}_A{\|S\|}_A$ and ${\|Tx\|}_A \leq {\|T\|}_A{\|x\|}_A$ for all $x\in \mathcal{H}$.

For proofs and more facts about this class of operators, we refer the reader to
\cite{Ar.Co.Go.1, Ar.Co.Go.2} and their references.

In recent years, several results covering some classes of operators
on a complex Hilbert space $\big(\mathcal{H}, \langle \cdot, \cdot\rangle\big)$
are extended to $\big(\mathcal{H}, {\langle \cdot, \cdot\rangle}_A\big)$
(see, e.g., \cite{Ar.Co.Go.2, B.F.O, Ba.Ka.Ah, Faghih, Fo.Go, Go, Wi.Se.Su, SA.S, Su, Z.3}).

The numerical radius of $T\in\mathbb{B}(\mathcal{H})$ is defined by
\begin{align*}
w(T) = \sup \Big\{|\langle Tx, x\rangle|:\,\,x\in\mathcal{H},\|x\| = 1\Big\}.
\end{align*}
This concept is useful in studying linear operators and has attracted the attention of many authors
in the last few decades (e.g., see \cite{D.2, G.R, K.1, M.S, Z.M.C.N}, and their references).

It is well known that $w(\cdot)$ defines a norm on
$\mathbb{B}(\mathcal{H})$ such that for all $T\in\mathbb{B}(\mathcal{H})$,
\begin{align}\label{1.1}
\frac{1}{2}\|T\| \leq w(T)\leq \|T\|.
\end{align}
The inequalities in (\ref{1.1}) are sharp. The first inequality becomes an equality if $T^{2}=0$.
The second inequality becomes an equality if $T$ is normal.

For more material about the numerical radius and
other results on numerical radius inequality,
see, e.g., \cite{H.K, G.R, K.1, K.M.Y}, and the references therein.

Some interesting numerical radius inequalities improving inequalities (\ref{1.1})
have been obtained by several mathematicians (see, e.g., \cite{A.K.1, A.K.2, D.1, Ga.Wu, H.K.S.2, K.M.Y, Y, Z.1}).

Motivated by theoretical study and applications, there have been many
generalizations of the numerical radius (e.g., see \cite{A.K.3, B.S, Bo.Ma, G.W.W, G.R, H.L.B, H.K.S.1, M.S, S.D.M, S.M.S, Sh, Z.2}).
One of these generalizations is the $A$-numerical radius of an operator $T\in \mathbb{B}(\mathcal{H})$ defined by
\begin{align*}
w_A(T) = \sup\Big\{\big|{\langle Tx, x\rangle}_A\big|: \,\,x\in \mathcal{H}, \,{\|x\|}_A = 1\Big\},
\end{align*}
see, e.g., \cite{Ba.Ka.Ah}.

Now, following by the Crawford number and the cosine of angle of an operator
$T\in \mathbb{B}(\mathcal{H})$ introduced by Gustafson and Rao in \cite{G.R},
we introduce the following notations
\begin{align*}
c_A(T) = \inf \Big\{|{\langle Tx, x\rangle}_A|:\,\,x\in\mathcal{H},{\|x\|}_A =1\Big\},
\end{align*}
\begin{equation*}
{|\cos|}_AT = \inf \left\{\frac{|{\langle Tx, x\rangle}_A|}{{\|Tx\|}_A{\|x\|}_A}: \,\, x\in \mathcal{H}, {\|Tx\|}_A{\|x\|}_A\neq 0\right\},
\end{equation*}
and
\begin{equation*}
{|\sin|}_AT = \sqrt{1 - {|\cos|}^2_AT}.
\end{equation*}

The paper is organized as follows.

In Section 2, inspired by the numerical radius inequalities of bounded linear operators
in \cite{A.K.1}, \cite{A.K.2}, \cite{H.K.S.2}, \cite{K.M.Y}, \cite{Z.1} and by using some ideas of them,
we first state a useful characterization
of the $A$-numerical radius for $T\in\mathbb{B}_{A}(\mathcal{H})$ as follows:
\begin{align*}
w_A(T) = \displaystyle{\sup_{\alpha^2 + \beta^2 = 1}}
{\left\|\alpha \frac{T + T^{\sharp_A}}{2} +
\beta \frac{T - T^{\sharp_A}}{2i}\right\|}_A.
\end{align*}
This expression was motivated by \cite[Theorem 2.1]{K.M.Y}.
We then apply it to find upper and lower bounds for
the $A$-numerical radius of semi-Hilbertian space operators.
Particularly, for $T\in\mathbb{B}_{A}(\mathcal{H})$ we prove that
\begin{align*}
w_A(T) \leq \frac{\sqrt{2}}{2}\sqrt{{\big\|T T^{\sharp_A} + T^{\sharp_A} T\big\|}_A} \leq {\|T\|}_A,
\end{align*}
\begin{align*}
\frac{1}{2}{\|T\|}_A\leq \sqrt{\frac{w^2_A(T)}{2} + \frac{w_A(T)}{2}\sqrt{w^2_A(T) - c^2_A(T)}} \leq w_A(T),
\end{align*}
and
\begin{align*}
\frac{1}{2}{\|T\|}_A \leq \max\Big\{{|\sin|}_AT, \frac{\sqrt{2}}{2}\Big\}w_A(T)\leq w_A(T).
\end{align*}

In Section 3, some upper bounds for the $A$-numerical
radius of products of semi-Hilbertian space operators are given.
In particular, for $T, S\in\mathbb{B}_{A}(\mathcal{H})$ we show that
\begin{align*}
w_A(TS) \leq w_A(T){\|S\|}_A + \frac{1}{2}w_A\Big((TS)^{\sharp_A} \pm T^{\sharp_A} S\Big) \leq 2w_A(T){\|S\|}_A.
\end{align*}

In the last section we present some upper bounds for the $A$-numerical radius of commutators
and anticommutators of semi-Hilbertian space operators.
Particularly, for $T, S\in\mathbb{B}_{A}(\mathcal{H})$ we prove that
\begin{align*}
w_A(TS^{\sharp_A} \pm ST^{\sharp_A}) \leq {\big\|T^{\sharp_A} T + SS^{\sharp_A}\big\|}_A.
\end{align*}

Our results generalize recent numerical radius inequalities of bounded linear operators
due to Kittaneh et al. \cite{A.K.1, A.K.2, H.K, K.1, K.M.Y, Z.1}.
\section{Upper and lower bounds of the $A$-numerical radius of operators}
We start our work with the following lemmas.
To establish the first lemma we use some ideas of \cite[Theorem 1.3-1]{G.R}.
\begin{lemma}\label{l.2.1}
Let $T\in\mathbb{B}_{A}(\mathcal{H})$ be an $A$-selfadjoint operator. Then
\begin{align*}
w_A(T) = {\|T\|}_A.
\end{align*}
\end{lemma}
\begin{proof}
Let $x \in \mathcal{H}$ with ${\|x\|}_A = 1$. By the Cauchy-Schwarz inequality, we have
\begin{align*}
\big|{\langle Tx, x\rangle}_A\big| \leq {\|Tx\|}_A {\|x\|}_A \leq {\|T\|}_A,
\end{align*}
and hence $w_A(T) = \sup\Big\{\big|{\langle Tx, x\rangle}_A\big|: \,\,x\in \mathcal{H}, \,{\|x\|}_A = 1\Big\}\leq {\|T\|}_A$.

Moreover, since $T$ is an $A$-selfadjoint operator,
for every $y, z \in \mathcal{H}$ such that ${\|y\|}_A = {\|z\|}_A = 1$ we have
\begin{align*}
{\langle T(y + z), y + z\rangle}_A = {\langle Ty, y\rangle}_A + 2\mbox{Re}{\langle Ty, z\rangle}_A + {\langle Tz, z\rangle}_A
\end{align*}
and
\begin{align*}
{\langle T(y - z), y - z\rangle}_A = {\langle Ty, y\rangle}_A - 2\mbox{Re}{\langle Ty, z\rangle}_A + {\langle Tz, z\rangle}_A.
\end{align*}
Consequently, we deduce
\begin{align*}
\mbox{Re}{\langle Ty, z\rangle}_A  = \frac{1}{4}\Big({\langle T(y + z), y + z\rangle}_A - {\langle T(y - z), y - z\rangle}_A\Big).
\end{align*}
So, we obtain
\begin{align*}
\big|\mbox{Re}{\langle Ty, z\rangle}_A\big| \leq \frac{w_A(T)}{4}\big({\|y + z\|}^2_A + {\|y - z\|}^2_A\big).
\end{align*}
Then it follows from parallelogram law that
\begin{align}\label{I.1.l.2.1}
\big|\mbox{Re}{\langle Ty, z\rangle}_A \big| \leq \frac{w_A(T)}{4}\big(2{\|y\|}^2_A + 2{\|z\|}^2_A\big) = w_A(T).
\end{align}
Now, consider the polar decomposition ${\langle Ty, z\rangle}_A = e^{i\theta}\big|{\langle Ty, z\rangle}_A\big|$
with $\theta \in \mathbb{R}$. By replacing $z$ by $e^{i\theta}z$ in (\ref{I.1.l.2.1}),
we get $\big|{\langle Ty, z\rangle}_A \big| = \mbox{Re}{\langle Ty, e^{i\theta}z\rangle}_A \leq w_A(T)$.
From this it follows that
${\|T\|}_A = \sup\Big\{\big|{\langle Ty, z\rangle}_A\big|: \,\,y, z\in \mathcal{H}, \,{\|y\|}_A = {\|z\|}_A = 1\Big\}\leq w_A(T)$
and consequently $w_A(T) = {\|T\|}_A$.
\end{proof}
\begin{remark}
Note that for an arbitrary operator $T$ of $\mathbb{B}_{A}(\mathcal{H})$, we have
\begin{align*}
0\leq {\|T\|}_A^2 - w_A^2(T) \leq
\inf_{\gamma \in \mathbb{C}}\Big\{{\|T + \gamma I\|}_A^2 - c_A^2(T + \gamma I)\Big\}.
\end{align*}
Indeed, if $x\in \mathcal{H}$ with ${\|x\|}_A= 1$, then simple computations show that
\begin{align*}
{\|Tx\|}^2_A - \big|{\langle Tx, x\rangle}_A\big|^2 =
{\|Tx - \gamma x\|}^2_A - \big|{\langle Tx - \gamma x, x\rangle}_A\big|^2 \qquad (\gamma \in \mathbb{C}),
\end{align*}
whence
\begin{align*}
{\|Tx\|}^2_A - \big|{\langle Tx, x\rangle}_A\big|^2 \leq
{\|T - \gamma I\|}^2_A - c^2_A(T - \gamma I) \qquad (\gamma \in \mathbb{C}).
\end{align*}
Thus
\begin{align*}
{\|Tx\|}^2_A - \big|{\langle Tx, x\rangle}_A\big|^2 \leq
\inf_{\gamma \in \mathbb{C}}\Big\{{\|T - \gamma I\|}^2_A - c^2_A(T - \gamma I)\Big\}.
\end{align*}
Taking the supremum in the above inequality over $x\in \mathcal{H}$, ${\|x\|}_A= 1$, we deduce the desired
inequality.
\end{remark}
The second lemma is stated as follows.
\begin{lemma}\label{l.2.1.5}
Let $T\in\mathbb{B}_{A}(\mathcal{H})$. For every $\theta \in \mathbb{R}$,
\begin{align*}
w_A\left(\frac{e^{i\theta}T + (e^{i\theta}T)^{\sharp_A}}{2}\right)
= {\left\|\frac{e^{i\theta}T + (e^{i\theta}T)^{\sharp_A}}{2}\right\|}_A.
\end{align*}
\end{lemma}
\begin{proof}
Let $\theta \in \mathbb{R}$. We have
$\Big(\frac{(e^{i\theta}T)^{\sharp_A} + \big((e^{i\theta}T)^{\sharp_A}\big)^{\sharp_A}}{2}\Big)^{\sharp_A}
= \frac{\big((e^{i\theta}T)^{\sharp_A}\big)^{\sharp_A} + (e^{i\theta}T)^{\sharp_A}}{2}$. Hence
$\frac{(e^{i\theta}T)^{\sharp_A} + \big((e^{i\theta}T)^{\sharp_A}\big)^{\sharp_A}}{2}$ is an $A$-selfadjoint operator.
So, by Lemma \ref{l.2.1} we get
\begin{align}\label{I.1.l.2.1.5}
w_A\left(\frac{(e^{i\theta}T)^{\sharp_A} + \big((e^{i\theta}T)^{\sharp_A}\big)^{\sharp_A}}{2}\right)
= {\left\|\frac{(e^{i\theta}T)^{\sharp_A} + \big((e^{i\theta}T)^{\sharp_A}\big)^{\sharp_A}}{2}\right\|}_A.
\end{align}
Since $w_A(R^{\sharp_A}) = w_A(R)$ and ${\|R^{\sharp_A}\|}_A = {\|R\|}_A$ for every $R\in\mathbb{B}_{A}(\mathcal{H})$,
from (\ref{I.1.l.2.1.5}) it follows that
\begin{align*}
w_A\left(\frac{e^{i\theta}T + (e^{i\theta}T)^{\sharp_A}}{2}\right)
= {\left\|\frac{e^{i\theta}T + (e^{i\theta}T)^{\sharp_A}}{2}\right\|}_A.
\end{align*}
\end{proof}
We now state the third lemma, which will be used to prove Theorem \ref{T.2.30}.
\begin{lemma}\label{l.2.2}
Let $T\in\mathbb{B}_{A}(\mathcal{H})$ and $x\in \mathcal{H}$. Then
\begin{align*}
\displaystyle{\sup_{\theta \in \mathbb{R}}}
\left|{\Big\langle \frac{e^{i\theta}T + (e^{i\theta}T)^{\sharp_A}}{2}x, x \Big\rangle}_A\right|
= |{\langle Tx, x\rangle}_A|.
\end{align*}
\end{lemma}
\begin{proof}
Let $\theta \in \mathbb{R}$. We have
\begin{align*}
\left|{\Big\langle \frac{e^{i\theta}T + (e^{i\theta}T)^{\sharp_A}}{2}x, x \Big\rangle}_A\right|
&= \frac{1}{2}\Big|e^{i\theta}{\langle Tx, x\rangle}_A + e^{-i\theta}{\langle T^{\sharp_A} x, x\rangle}_A\Big|
\\& = \frac{1}{2}\Big|e^{i\theta}{\langle Tx, x\rangle}_A + e^{-i\theta}{\langle x, Tx\rangle}_A\Big|.
\end{align*}
Thus
\begin{align}\label{I.1.l.2.2}
\left|{\Big\langle \frac{e^{i\theta}T + (e^{i\theta}T)^{\sharp_A}}{2}x, x \Big\rangle}_A\right|
= \Big|\mbox{Re}\big(e^{i\theta}{\langle Tx, x\rangle}_A\big)\Big|.
\end{align}
From this it follows that
\begin{align*}
\left|{\Big\langle \frac{e^{i\theta}T + (e^{i\theta}T)^{\sharp_A}}{2}x, x \Big\rangle}_A\right|
\leq \big|{\langle Tx, x\rangle}_A\big|,
\end{align*}
whence
\begin{align}\label{I.2.l.2.2}
\displaystyle{\sup_{\theta \in \mathbb{R}}}
\left|{\Big\langle \frac{e^{i\theta}T + (e^{i\theta}T)^{\sharp_A}}{2}x, x \Big\rangle}_A\right|
\leq |{\langle Tx, x\rangle}_A|.
\end{align}
Now, if $|{\langle Tx, x\rangle}_A| = 0$, then from (\ref{I.1.l.2.2}) we obtain
$\left|{\Big\langle \frac{e^{i\theta}T + (e^{i\theta}T)^{\sharp_A}}{2}x, x \Big\rangle}_A\right| = 0$
and so
\begin{align*}
\displaystyle{\sup_{\theta \in \mathbb{R}}}
\left|{\Big\langle \frac{e^{i\theta}T + (e^{i\theta}T)^{\sharp_A}}{2}x, x \Big\rangle}_A\right| = 0 = |{\langle Tx, x\rangle}_A|.
\end{align*}
If $|{\langle Tx, x\rangle}_A| \neq 0$, then we put $e^{i\theta_0} = \frac{{\langle x, Tx\rangle}_A}{|{\langle Tx, x\rangle}_A|}$.
Therefore, by (\ref{I.1.l.2.2}), we obtain
\begin{align}\label{I.3.l.2.2}
\left|{\Big\langle \frac{e^{i\theta_0}T + (e^{i\theta_0}T)^{\sharp_A}}{2}x, x \Big\rangle}_A\right|
= \Big|\mbox{Re}\big(e^{i\theta_0}{\langle Tx, x\rangle}_A\big)\Big| = |{\langle Tx, x\rangle}_A|.
\end{align}
From (\ref{I.2.l.2.2}) and (\ref{I.3.l.2.2}) it follows that
\begin{align*}
\displaystyle{\sup_{\theta \in \mathbb{R}}}
\left|{\Big\langle \frac{e^{i\theta}T + (e^{i\theta}T)^{\sharp_A}}{2}x, x \Big\rangle}_A\right|
= |{\langle Tx, x\rangle}_A|.
\end{align*}
\end{proof}
Now, we are in a position to state a useful characterization of the $A$-numerical
radius for semi-Hilbertian space operators.
\begin{theorem}\label{T.2.30}
Let $T\in\mathbb{B}_{A}(\mathcal{H})$. Then
\begin{align*}
w_A(T) = \displaystyle{\sup_{\theta \in \mathbb{R}}}{\left\|\frac{e^{i\theta}T + (e^{i\theta}T)^{\sharp_A}}{2}\right\|}_A.
\end{align*}
\end{theorem}
\begin{proof}
Let $\theta \in \mathbb{R}$. By Lemma \ref{l.2.1.5} it follows that
\begin{align*}
w_A\left(\frac{e^{i\theta}T + (e^{i\theta}T)^{\sharp_A}}{2}\right)
= {\left\|\frac{e^{i\theta}T + (e^{i\theta}T)^{\sharp_A}}{2}\right\|}_A.
\end{align*}
Therefore, by Lemma \ref{l.2.2} we conclude that
\begin{align*}
\displaystyle{\sup_{\theta \in \mathbb{R}}}{\left\|\frac{e^{i\theta}T + (e^{i\theta}T)^{\sharp_A}}{2}\right\|}_A
&= \displaystyle{\sup_{\theta \in \mathbb{R}}}\,w_A\left(\frac{e^{i\theta}T + (e^{i\theta}T)^{\sharp_A}}{2}\right)
\\&= \displaystyle{\sup_{\theta \in \mathbb{R}}}\,\displaystyle{\sup_{{\|x\|}_A = 1}}\left|{\Big\langle \frac{e^{i\theta}T
+ (e^{i\theta}T)^{\sharp_A}}{2}x, x \Big\rangle}_A\right|
\\&= \displaystyle{\sup_{{\|x\|}_A = 1}}\big|{\langle Tx, x\rangle}_A\big| = w_A(T).
\end{align*}
\end{proof}
Here we present one of the main results of this section.
\begin{theorem}\label{T.2.4}
Let $T\in\mathbb{B}_{A}(\mathcal{H})$.
Then for $\alpha, \beta \in \mathbb{R}$,
\begin{align*}
w_A(T) = \displaystyle{\sup_{\alpha^2 + \beta^2 = 1}}
{\left\|\alpha \frac{T + T^{\sharp_A}}{2} +
\beta \frac{T - T^{\sharp_A}}{2i}\right\|}_A.
\end{align*}
\end{theorem}
\begin{proof}
Let $\theta \in \mathbb{R}$. Put $\alpha = \cos \theta$ and $\beta = -\sin \theta$.
We have
\begin{align*}
\frac{e^{i\theta}T + (e^{i\theta}T)^{\sharp_A}}{2}
&= \frac{(\cos \theta + i\sin \theta)T + (\cos \theta - i\sin \theta)T^{\sharp_A}}{2}
\\& = \cos \theta \frac{T + T^{\sharp_A}}{2} - \sin \theta \frac{T - T^{\sharp_A}}{2i}
\\& = \alpha \frac{T + T^{\sharp_A}}{2} + \beta \frac{T - T^{\sharp_A}}{2i}.
\end{align*}
Therefore
\begin{align*}
\displaystyle{\sup_{\theta \in \mathbb{R}}}
{\left\|\frac{e^{i\theta}T + (e^{i\theta}T)^{\sharp_A}}{2}\right\|}_A
= \displaystyle{\sup_{\alpha^2 + \beta^2 = 1}}
{\left\|\alpha \frac{T + T^{\sharp_A}}{2} +
\beta \frac{T - T^{\sharp_A}}{2i}\right\|}_A,
\end{align*}
and hence, by Theorem \ref{T.2.30}, we obtain
\begin{align*}
w_A(T) = \displaystyle{\sup_{\alpha^2 + \beta^2 = 1}}
{\left\|\alpha \frac{T + T^{\sharp_A}}{2} +
\beta \frac{T - T^{\sharp_A}}{2i}\right\|}_A.
\end{align*}
\end{proof}
As a consequence of Theorem \ref{T.2.4}, we have the following result.
\begin{corollary}\label{C.2.5}
Let $T\in\mathbb{B}_{A}(\mathcal{H})$. Then
\begin{align*}
\max\left\{{\left\|\frac{T + T^{\sharp_A}}{2}\right\|}_A
, {\left\|\frac{T - T^{\sharp_A}}{2i}\right\|}_A\right\}\leq w_A(T).
\end{align*}
\end{corollary}
\begin{proof}
By setting $(\alpha, \beta) =(1, 0)$ and $(\alpha, \beta) =(0, 1)$ in Theorem \ref{T.2.4}, the result follows.
\end{proof}
The following result is another consequence of Theorem \ref{T.2.4}.
\begin{corollary}\label{C.2.6}
Let $T\in\mathbb{B}_{A}(\mathcal{H})$. Then
\begin{align}\label{I.1.C.2.6}
\frac{1}{2}{\|T\|}_A\leq w_A(T) \leq {\|T\|}_A.
\end{align}
\end{corollary}
\begin{proof}
Clearly, $w_A(T) \leq {\|T\|}_A$.
On the other hand, by using Corollary \ref{C.2.5}, we get
\begin{align*}
{\|T\|}_A = {\left\|\frac{T + T^{\sharp_A}}{2} + i\frac{T - T^{\sharp_A} }{2i}\right\|}_A
\leq {\left\|\frac{T + T^{\sharp_A}}{2}\right\|}_A + {\left\|\frac{T - T^{\sharp_A}}{2i}\right\|}_A
\leq 2w_A(T).
\end{align*}
Hence $\frac{1}{2}{\|T\|}_A\leq w_A(T)$.
\end{proof}
\begin{remark}\label{R.2.6.5}
Corollary \ref{C.2.6} has recently been proved by Baklouti et al. in \cite{Ba.Ka.Ah}.
Our approach here is different from theirs.
\end{remark}
In the following theorem, we give a improvement of the second inequality in (\ref{I.1.C.2.6}).
\begin{theorem}\label{T.2.7}
Let $T\in\mathbb{B}_{A}(\mathcal{H})$. Then
\begin{align*}
w_A(T) \leq \frac{\sqrt{2}}{2}\sqrt{{\big\|T T^{\sharp_A} + T^{\sharp_A} T\big\|}_A} \leq {\|T\|}_A.
\end{align*}
\end{theorem}
\begin{proof}
Put $M : = \frac{T^{\sharp_A} + (T^{\sharp_A})^{\sharp_A}}{2}$ and
$N : = \frac{T^{\sharp_A} - (T^{\sharp_A})^{\sharp_A}}{2i}$.
Then $T^{\sharp_A} = M + iN$. Also, simple computations show that
\begin{align*}
M^2 + N^2
= \frac{(T^{\sharp_A})^{\sharp_A}T^{\sharp_A} + T^{\sharp_A}(T^{\sharp_A})^{\sharp_A}}{2}
= \left(\frac{TT^{\sharp_A} + T^{\sharp_A} T}{2}\right)^{\sharp_A}.
\end{align*}
Since ${\|R^{\sharp_A}\|}_A = {\|R\|}_A$ for every $R\in\mathbb{B}_{A}(\mathcal{H})$, hence
\begin{align}\label{I.0.T.2.7}
{\|M^2 + N^2\|}_A
= \frac{1}{2}{\Big\|TT^{\sharp_A} + T^{\sharp_A} T\Big\|}_A.
\end{align}
Now, let $x\in \mathcal{H}$ with ${\|x\|}_A = 1$. We have
\begin{align*}
\big|{\langle x, Tx \rangle}_A\big|^2 &= \big|{\langle T^{\sharp_A}x, x \rangle}_A\big|^2
\\&= {\big\langle (M + i N)x, x\big\rangle}_A
{\big\langle x, (M + i N)x\big\rangle}_A
\\& = \Big({\langle Mx, x\rangle}_A + i{\langle Nx, x\rangle}_A \Big)
\Big({\langle x, Mx\rangle}_A - i{\langle x, Nx\rangle}_A \Big)
\\& = \big|{\langle Mx, x\rangle}_A\big|^2 + \big|{\langle Nx, x\rangle}_A\big|^2
\\& \leq {\langle Mx, Mx\rangle}_A + {\langle Nx, Nx\rangle}_A \qquad\big(\mbox{by the Cauchy-Schwarz inequality}\big)
\\& = {\langle M^2x, x\rangle}_A + {\langle N^2x, x\rangle}_A \qquad \qquad\big(\mbox{since $M^{\sharp_A} = M$ and $N^{\sharp_A} = N$}\big)
\\& = {\big\langle (M^2 + N^2)x, x\big\rangle}_A
\\& \leq {\|M^2 + N^2\|}_A = \frac{1}{2}{\Big\|TT^{\sharp_A} + T^{\sharp_A} T\Big\|}_A. \qquad \qquad\big(\mbox{by (\ref{I.0.T.2.7})}\big)
\end{align*}
Hence
\begin{align*}
w^2_A(T) = \displaystyle{\sup_{{\|x\|}_A = 1}}
\big|{\langle x, Tx \rangle}_A\big|^2
\leq \frac{1}{2}{\Big\|TT^{\sharp_A} + T^{\sharp_A} T\Big\|}_A,
\end{align*}
or equivalently,
\begin{align}\label{I.1.T.2.7}
w_A(T) \leq \frac{\sqrt{2}}{2}\sqrt{{\big\|T T^{\sharp_A} + T^{\sharp_A} T\big\|}_A}.
\end{align}
Further, since ${\|T T^{\sharp_A}\|}_A = {\|T^{\sharp_A} T\|}_A = {\|T\|}^2_A$, by the triangle inequality we obtain
\begin{align*}
\frac{\sqrt{2}}{2}\sqrt{{\big\|T T^{\sharp_A} + T^{\sharp_A} T\big\|}_A}
\leq \frac{\sqrt{2}}{2}\sqrt{{\|T T^{\sharp_A}\|}_A + {\|T^{\sharp_A} T\|}_A}
= {\|T\|}_A.
\end{align*}
Thus
\begin{align}\label{I.2.T.2.7}
\frac{\sqrt{2}}{2}\sqrt{{\big\|T T^{\sharp_A} + T^{\sharp_A} T\big\|}_A} \leq {\|T\|}_A.
\end{align}
By (\ref{I.1.T.2.7}) and (\ref{I.2.T.2.7}) we deduce the desired result.
\end{proof}
Next, we present another improvement of the second inequality in (\ref{I.1.C.2.6}).
\begin{theorem}\label{T.2.3}
Let $T\in\mathbb{B}_{A}(\mathcal{H})$. Then
\begin{align*}
w_A(T) \leq \frac{1}{2}\sqrt{{\big\|TT^{\sharp_A} + T^{\sharp_A} T\big\|}_A + 2w_A(T^2)} \leq {\|T\|}_A.
\end{align*}
\end{theorem}
\begin{proof}
By Theorem \ref{T.2.30}, we have
\begin{align*}
w_A(T) & = \displaystyle{\sup_{\theta \in \mathbb{R}}}{\left\|\frac{e^{i\theta}T + (e^{i\theta}T)^{\sharp_A}}{2}\right\|}_A
\\& = \frac{1}{2}\displaystyle{\sup_{\theta \in \mathbb{R}}}{\left\|(e^{i\theta}T)^{\sharp_A} + \big((e^{i\theta}T)^{\sharp_A}\big)^{\sharp_A}\right\|}_A
\\& \qquad \qquad \qquad \qquad \qquad\big(\mbox{since ${\|R\|}_A = {\|R^{\sharp_A}\|}_A$ for every $R\in\mathbb{B}_{A}(\mathcal{H})$}\big)
\\& = \frac{1}{2}\displaystyle{\sup_{\theta \in \mathbb{R}}}
\sqrt{{\left\|\left((e^{i\theta}T)^{\sharp_A} + \big((e^{i\theta}T)^{\sharp_A}\big)^{\sharp_A}\right)
\left(\big((e^{i\theta}T)^{\sharp_A}\big)^{\sharp_A} + (e^{i\theta}T)^{\sharp_A}\right)\right\|}_A}
\\& \qquad \qquad \qquad \qquad \qquad\big(\mbox{since ${\|R\|}^2_A = {\|RR^{\sharp_A}\|}_A$ for every $R\in\mathbb{B}_{A}(\mathcal{H})$}\big)
\\& = \frac{1}{2}\displaystyle{\sup_{\theta \in \mathbb{R}}}
\sqrt{{\left\|T^{\sharp_A} (T^{\sharp_A})^{\sharp_A} + (T^{\sharp_A})^{\sharp_A} T^{\sharp_A}
+ \big((e^{i\theta}T)^{\sharp_A}\big)^2 + \Big(\big((e^{i\theta}T)^{\sharp_A}\big)^{\sharp_A}\Big)^2\right\|}_A}
\\& \leq \frac{1}{2}\displaystyle{\sup_{\theta \in \mathbb{R}}}
\sqrt{{\Big\|T^{\sharp_A} (T^{\sharp_A})^{\sharp_A} + (T^{\sharp_A})^{\sharp_A} T^{\sharp_A} \Big\|}_A
+ {\left\|\big((e^{i\theta}T)^{\sharp_A}\big)^2 + \Big(\big((e^{i\theta}T)^{\sharp_A}\big)^{\sharp_A}\Big)^2\right\|}_A}
\\& \leq \frac{1}{2}\displaystyle{\sup_{\theta \in \mathbb{R}}}
\sqrt{{\Big\|TT^{\sharp_A} + T^{\sharp_A} T \Big\|}_A
+ 2{\left\|\frac{e^{2i\theta}T^2 + \big(e^{2i\theta}T^2\big)^{\sharp_A}}{2}\right\|}_A}
\\& \qquad \qquad \qquad \qquad \qquad\big(\mbox{since ${\|R^{\sharp_A}\|}_A = {\|R\|}_A$ for every $R\in\mathbb{B}_{A}(\mathcal{H})$}\big)
\\& \leq \frac{1}{2}\sqrt{{\Big\|TT^{\sharp_A} + T^{\sharp_A} T \Big\|}_A
+ 2\displaystyle{\sup_{\theta \in \mathbb{R}}}{\left\|\frac{e^{2i\theta}T^2 + \big(e^{2i\theta}T^2\big)^{\sharp_A}}{2}\right\|}_A}
\\& = \frac{1}{2}\sqrt{{\Big\|TT^{\sharp_A} + T^{\sharp_A} T \Big\|}_A + 2w_A(T^2)}
\qquad \qquad\big(\mbox{by Theorem \ref{T.2.30}}\big)
\\& \leq \frac{1}{2}\sqrt{{\|TT^{\sharp_A}\|}_A + {\|T^{\sharp_A} T\|}_A + 2w_A(T^2)}
\\& = \frac{\sqrt{2}}{2}\sqrt{{\|T\|}^2_A + w_A(T^2)}\qquad\big(\mbox{since ${\|RR^{\sharp_A}\|}_A = {\|R\|}^2_A$ for every $R\in\mathbb{B}_{A}(\mathcal{H})$}\big)
\\& \leq \frac{\sqrt{2}}{2}\sqrt{{\|T\|}^2_A + {\|T^2\|}_A} \qquad \qquad\big(\mbox{by Corollary \ref{C.2.6}}\big)
\\& \leq \frac{\sqrt{2}}{2}\sqrt{{\|T\|}^2_A + {\|T\|}^2_A} = {\|T\|}_A,
\end{align*}
which proves the desired inequalities.
\end{proof}
In the following theorem, we establish an improvement of the first inequality in (\ref{I.1.C.2.6}).
\begin{theorem}\label{T.2.3}
Let $T\in\mathbb{B}_{A}(\mathcal{H})$. Then
\begin{align*}
\frac{1}{2}{\|T\|}_A \leq \frac{1}{2}\sqrt{{\big\|TT^{\sharp_A} + T^{\sharp_A} T\big\|}_A + 2c_A(T^2)} \leq w_A(T).
\end{align*}
\end{theorem}
\begin{proof}
Let $x\in \mathcal{H}$ with ${\|x\|}_A = 1$.
Suppose that $\big|{\langle T^{\sharp_A} T^{\sharp_A} x, x\rangle}_A\big| = e^{-2i\theta}{\langle T^{\sharp_A} T^{\sharp_A} x, x\rangle}_A$
for some real number $\theta$.
Then, we have
\begin{align*}
e^{2i\theta}{\big\langle(T^{\sharp_A})^{\sharp_A}(T^{\sharp_A})^{\sharp_A} x, x \big\rangle}_A
= e^{2i\theta} \overline{{\langle T^{\sharp_A} T^{\sharp_A} x, x \rangle}_A}
= \big|{\langle T^{\sharp_A} T^{\sharp_A} x, x\rangle}_A\big|
= \big|{\langle x, T^2x\rangle}_A\big|.
\end{align*}
Thus
\begin{align}\label{I.100.T.2.3}
e^{-2i\theta} {\langle T^{\sharp_A} T^{\sharp_A} x, x \rangle}_A = \big|{\langle x, T^2x\rangle}_A\big| = e^{2i\theta}{\big\langle(T^{\sharp_A})^{\sharp_A}(T^{\sharp_A})^{\sharp_A} x, x \big\rangle}_A.
\end{align}
So, by Theorem \ref{T.2.30}, we obtain
\begin{align*}
4w^2_A(T) & \geq {\left\|e^{i\theta}T + (e^{i\theta}T)^{\sharp_A}\right\|}^2_A
\\&= {\left\|(e^{i\theta}T)^{\sharp_A} + \big((e^{i\theta}T)^{\sharp_A}\big)^{\sharp_A}\right\|}^2_A
\\& \qquad \qquad \qquad \qquad \qquad\big(\mbox{since ${\|R\|}_A = {\|R^{\sharp_A}\|}_A$ for every $R\in\mathbb{B}_{A}(\mathcal{H})$}\big)
\\& = {\left\|\left((e^{i\theta}T)^{\sharp_A} + \big((e^{i\theta}T)^{\sharp_A}\big)^{\sharp_A}\right)
\left((e^{i\theta}T)^{\sharp_A} + \big((e^{i\theta}T)^{\sharp_A}\big)^{\sharp_A}\right)^{\sharp_A}\right\|}_A
\\& \qquad \qquad \qquad \qquad \qquad\big(\mbox{since ${\|R\|}^2_A = {\|RR^{\sharp_A}\|}_A$ for every $R\in\mathbb{B}_{A}(\mathcal{H})$}\big)
\\& = {\Big\|T^{\sharp_A} (T^{\sharp_A})^{\sharp_A} + (T^{\sharp_A})^{\sharp_A} T^{\sharp_A}
+ e^{-2i\theta}T^{\sharp_A} T^{\sharp_A} + e^{2i\theta}(T^{\sharp_A})^{\sharp_A}(T^{\sharp_A})^{\sharp_A}\Big\|}_A
\\& \qquad \qquad \qquad \qquad \qquad\big(\mbox{since $\big((R^{\sharp_A})^{\sharp_A}\big)^{\sharp_A} = R^{\sharp_A}$ for every $R\in\mathbb{B}_{A}(\mathcal{H})$}\big)
\\& \geq \Big|{\Big \langle \Big(T^{\sharp_A} (T^{\sharp_A})^{\sharp_A} + (T^{\sharp_A})^{\sharp_A} T^{\sharp_A}
+ e^{-2i\theta}T^{\sharp_A} T^{\sharp_A} + e^{2i\theta}(T^{\sharp_A})^{\sharp_A}(T^{\sharp_A})^{\sharp_A}\Big)x, x\Big \rangle}_A\Big|
\\& = \Big|{\Big \langle \Big(T^{\sharp_A} (T^{\sharp_A})^{\sharp_A} + (T^{\sharp_A})^{\sharp_A} T^{\sharp_A}\Big)x, x\Big \rangle}_A
\\& \qquad \qquad \qquad \qquad + e^{-2i\theta} {\langle T^{\sharp_A} T^{\sharp_A} x, x \rangle}_A
+ e^{2i\theta}{\big\langle(T^{\sharp_A})^{\sharp_A}(T^{\sharp_A})^{\sharp_A} x, x \big\rangle}_A\Big|
\\& = \Big {\langle \Big(T^{\sharp_A} (T^{\sharp_A})^{\sharp_A} + (T^{\sharp_A})^{\sharp_A} T^{\sharp_A}\Big)x, x\Big \rangle}_A
+ 2\big|{\langle x, T^2x \rangle}_A\big| \qquad \qquad\big(\mbox{by \ref{I.100.T.2.3}}\big)
\\& \geq \Big {\langle \Big(T^{\sharp_A} (T^{\sharp_A})^{\sharp_A} + (T^{\sharp_A})^{\sharp_A} T^{\sharp_A}\Big)x, x\Big \rangle}_A
+ 2c_A(T^2).
\end{align*}
From this it follows that
\begin{align*}
\frac{1}{2}\sqrt{\Big {\langle \Big(T^{\sharp_A} (T^{\sharp_A})^{\sharp_A} + (T^{\sharp_A})^{\sharp_A} T^{\sharp_A}\Big)x, x\Big \rangle}_A
+ 2c_A(T^2)} \leq w_A(T).
\end{align*}
Taking the supremum over $x \in \mathcal{H}$ with ${\|x\|}_A = 1$ in the above inequality we get
\begin{align}\label{I.3.T.2.3}
\frac{1}{2}\sqrt{{\big\|TT^{\sharp_A} + T^{\sharp_A} T\big\|}_A + 2c_A(T^2)} \leq w_A(T).
\end{align}
Furthermore, since $T^{\sharp_A} T$ is an $A$-positive operator,
from ${\big\|TT^{\sharp_A} + T^{\sharp_A} T\big\|}_A \geq {\|TT^{\sharp_A}\|}_A = {\|T\|}^2_A$ it follows that
\begin{align}\label{I.4.T.2.3}
\frac{1}{2}\sqrt{{\big\|TT^{\sharp_A} + T^{\sharp_A} T\big\|}_A + 2c_A(T^2)}
\geq \frac{1}{2}\sqrt{{\big\|TT^{\sharp_A} + T^{\sharp_A} T\big\|}_A} \geq \frac{1}{2}{\|T\|}_A.
\end{align}
Now, by (\ref{I.3.T.2.3}) and (\ref{I.4.T.2.3}) we conclude that
\begin{align*}
\frac{1}{2}{\|T\|}_A \leq \frac{1}{2}\sqrt{{\big\|TT^{\sharp_A} + T^{\sharp_A} T\big\|}_A + 2c_A(T^2)} \leq w_A(T).
\end{align*}
\end{proof}
Now, we give another improvement of the first inequality in (\ref{I.1.C.2.6}).
\begin{theorem}\label{T.2.3}
Let $T\in\mathbb{B}_{A}(\mathcal{H})$. Then
\begin{align*}
\frac{1}{2}{\|T\|}_A\leq \sqrt{\frac{w^2_A(T)}{2} + \frac{w_A(T)}{2}\sqrt{w^2_A(T) - c^2_A(T)}} \leq w_A(T).
\end{align*}
\end{theorem}
\begin{proof}
Clearly, $\sqrt{\frac{w^2_A(T)}{2} + \frac{w_A(T)}{2}\sqrt{w^2_A(T) - c^2_A(T)}} \leq w_A(T)$.
Now, let $x\in \mathcal{H}$ with ${\|x\|}_A = 1$.
Suppose that $\big|{\langle Tx, x\rangle}_A\big| = e^{i\theta}{\langle Tx, x\rangle}_A$
for some real number $\theta$.
Put $M : = \frac{e^{i\theta}T + (e^{i\theta}T)^{\sharp_A}}{2}$ and
$N : = \frac{e^{i\theta}T - (e^{i\theta}T)^{\sharp_A}}{2i}$.
Then $M + iN = e^{i\theta}T$ and
\begin{align*}
{\langle Mx, x\rangle}_A + i{\langle Nx, x\rangle}_A = {\langle e^{i\theta}Tx, x\rangle}_A = \big|{\langle Tx, x\rangle}_A\big|\geq 0.
\end{align*}
It follows from ${\langle Nx, x\rangle}_A = \mbox{Im}{\langle e^{i\theta}Tx, x\rangle}_A \in\mathbb{R}$ that
\begin{align*}
{\langle e^{i\theta}Tx, x\rangle}_A = {\langle Mx, x\rangle}_A,
\qquad {\langle Nx, x\rangle}_A = 0.
\end{align*}
So, we have
\begin{align*}
\frac{1}{4}{\|Tx\|}^2_A & =
\frac{1}{4}\left({\Big\|e^{i\theta}Tx - {\langle e^{i\theta}Tx, x\rangle}_A x \Big\|}^2_A
+ |{\langle Tx, x\rangle}_A|^2\right)\\
& = \frac{1}{4}\left({\Big\| Mx - {\langle Mx, x\rangle}_A x + iNx \Big\|}^2_A
+ |{\langle Tx, x\rangle}_A|^2\right)
\\& \qquad \qquad \qquad \qquad \qquad \qquad \qquad \qquad \qquad \qquad
\text{\big(since }{\langle Nx, x\rangle}_A = 0\text{\big)}\\
& \leq \frac{1}{4}\left(\Big(\Big{\|Mx - {\langle Mx, x\rangle}_A x \Big\|}_A + {\|Nx\|}_A\Big)^2
+ |{\langle Tx, x\rangle}_A|^2\right)\\
& = \frac{1}{4}\left(\left(\sqrt{{\|Mx\|}^2_A - |{\langle Mx, x\rangle}_A|^2} + {\|Nx\|}_A\right)^2
+ |{\langle Tx, x\rangle}_A|^2\right)\\
& = \frac{1}{4}\left(\left(\sqrt{{\|Mx\|}^2_A - |{\langle e^{i\theta}Tx, x\rangle}_A|^2} + {\|Nx\|}_A\right)^2
+ |{\langle Tx, x\rangle}_A|^2\right)\\
&\qquad \qquad \qquad \qquad \qquad \qquad \qquad \qquad \qquad
\text{\big(since }{\langle Mx, x\rangle}_A = {\langle e^{i\theta}Tx, x\rangle}_A\text{\big)}\\
&\leq \frac{1}{4}\left(\left(\sqrt{{\|M\|}^2_A - |{\langle Tx, x\rangle}_A|^2} + {\|N\|}_A\right)^2 + |{\langle Tx, x\rangle}_A|^2\right)\\
&\leq \frac{1}{4}\left(\left(\sqrt{w^2_A(T) - |{\langle Tx, x\rangle}_A|^2} + w_A(T)\right)^2 + |{\langle Tx, x\rangle}_A|^2\right)\\
&\qquad \qquad \qquad \qquad \qquad \qquad \qquad
\text{\big(since }{\|M\|}_A, {\|N\|}_A\leq w_A(e^{i\theta}T) = w_A(T)\text{\big)}\\
& = \frac{w^2_A(T)}{2} + \frac{w_A(T)}{2}\sqrt{w^2_A(T) - |{\langle Tx, x\rangle}_A|^2}.
\end{align*}
Hence
\begin{align}\label{I.1.T.2.3}
\frac{1}{2}{\|Tx\|}_A\leq \sqrt{\frac{w^2_A(T)}{2} + \frac{w_A(T)}{2}\sqrt{w^2_A(T) - |{\langle Tx, x\rangle}_A|^2}} \qquad ({\|x\|}_A = 1),
\end{align}
which implies
\begin{align*}
\frac{1}{2}{\|Tx\|}_A\leq \sqrt{\frac{w^2_A(T)}{2} + \frac{w_A(T)}{2}\sqrt{w^2_A(T) - c^2_A(T)}}.
\end{align*}
Taking the supremum over $x\in\mathcal{H}$ with ${\|x\|}_A = 1$
in the above inequality we get
\begin{align*}
\frac{1}{2}{\|T\|}_A\leq \sqrt{\frac{w^2_A(T)}{2} + \frac{w_A(T)}{2}\sqrt{w^2_A(T) - c^2_A(T)}}.
\end{align*}
\end{proof}
We end this section with a considerable improvement of the first inequality in (\ref{I.1.C.2.6}).
\begin{theorem}\label{T.2.12}
Let $T\in\mathbb{B}_{A}(\mathcal{H})$. Then
\begin{align*}
\frac{1}{2}{\|T\|}_A \leq \max\Big\{|\sin|_AT, \frac{\sqrt{2}}{2}\Big\}w_A(T)\leq w_A(T).
\end{align*}
\end{theorem}
\begin{proof}
Obviously,
\begin{align*}
\max\Big\{{|\sin|}_AT, \frac{\sqrt{2}}{2}\Big\}w_A(T)\leq w_A(T).
\end{align*}
Furthermore, by (\ref{I.1.T.2.3}) we have
\begin{align*}
\frac{1}{2}{\|Tx\|}_A\leq \sqrt{\frac{w^2_A(T)}{2} + \frac{w_A(T)}{2}\sqrt{w^2_A(T) - |{\langle Tx, x\rangle}_A|^2}} \qquad ({\|x\|}_A = 1),
\end{align*}
and hence
\begin{align*}
\frac{1}{2}{\|Tx\|}_A\leq \sqrt{\frac{w^2_A(T)}{2} + \frac{w_A(T)}{2}\sqrt{w^2_A(T) - {\|Tx\|}^2_A{|\cos|}^2_AT}}.
\end{align*}
From this it follows that
\begin{align}\label{I.1.T.2.12}
{\|Tx\|}^2_A - 2w^2_A(T)\leq 2w_A(T)\sqrt{w^2_A(T) - {\|Tx\|}^2_A{|\cos|}^2_AT}.
\end{align}
We consider two cases.\\

Case 1. ${\|Tx\|}^2_A - 2w^2_A(T)\leq 0$. Then we reach that ${\|Tx\|}_A \leq \sqrt{2}w_A(T)$ and so
\begin{align}\label{I.2.T.2.12}
\frac{1}{2}{\|T\|}_A \leq \frac{\sqrt{2}}{2}w_A(T).
\end{align}

Case 2. ${\|Tx\|}^2_A - 2w^2_A(T) > 0$. By (\ref{I.1.T.2.12}) it follows that
\begin{align*}
{\| Tx\|}^4_A - 4{\|Tx\|}^2_Aw^2_A(T) + 4w^4_A(T)\leq 4w^4_A(T) - 4w^2_A(T){\|Tx\|}^2_A{|\cos|}^2_AT.
\end{align*}
Thus
\begin{align*}
{\|Tx\|}^2_A \leq 4\left(1 - {|\cos|}^2_AT\right)w^2_A(T).
\end{align*}
This yields
\begin{align*}
\frac{1}{2}{\|Tx\|}_A \leq |\sin|_ATw_A(T),
\end{align*}
and hence
\begin{align}\label{I.3.T.2.12}
\frac{1}{2}{\|T\|}_A \leq |\sin|_ATw_A(T).
\end{align}
Now, by (\ref{I.2.T.2.12}) and (\ref{I.3.T.2.12}) we obtain
\begin{align*}
\frac{1}{2}{\|T\|}_A \leq \max\Big\{|\sin|_AT, \frac{\sqrt{2}}{2}\Big\}w_A(T).
\end{align*}
\end{proof}
\section{Upper bounds for the $A$-numerical radius of products of operators}
In this section, we derive upper bounds for the $A$-numerical radius of products of semi-Hilbertian space operators.
Since for every $T, S\in\mathbb{B}_{A}(\mathcal{H})$ we have ${\|TS\|}_A \leq {\|T\|}_A {\|S\|}_A$,
by the inequalities of (\ref{I.1.C.2.6}) we obtain
\begin{align}\label{I.10.1}
w_A(TS ) \leq {\|TS\|}_A \leq 2{\|T\|}_Aw_A(S) \leq 4w_A(T)w_A(S).
\end{align}
In the following theorems, we improve the inequalities \ref{I.10.1}.
To achieve our goal, we need the following lemma.
\begin{lemma}\label{L.2.10}
Let $T, S\in\mathbb{B}_{A}(\mathcal{H})$. Then
\begin{align*}
w_A\Big((TS)^{\sharp_A} \pm T^{\sharp_A} S\Big)\leq 2w_A(T){\|S\|}_A .
\end{align*}
\end{lemma}
\begin{proof}
Let $\theta \in \mathbb{R}$. Since $\big((R^{\sharp_A})^{\sharp_A}\big)^{\sharp_A} = R^{\sharp_A}$
for every $R\in\mathbb{B}_{A}(\mathcal{H})$,
we have
\begin{align}\label{I.1.L.2.30}
&\frac{e^{i\theta}\big(((TS)^{\sharp_A} + T^{\sharp_A} S)\big)^{\sharp_A} + \Big(e^{i\theta}\big(((TS)^{\sharp_A} + T^{\sharp_A} S)\big)^{\sharp_A}\Big)^{\sharp_A}}{2} \nonumber
\\& \qquad = \frac{e^{i\theta}\big(T^{\sharp_A}\big)^{\sharp_A} \big(S^{\sharp_A}\big)^{\sharp_A} + e^{i\theta}S^{\sharp_A} \big(T^{\sharp_A}\big)^{\sharp_A}
+ e^{-i\theta} S^{\sharp_A} T^{\sharp_A} + e^{-i\theta} T^{\sharp_A} \big(S^{\sharp_A}\big)^{\sharp_A}}{2} \nonumber
\\& \qquad \qquad = \frac{e^{-i\theta} T^{\sharp_A} + \big(e^{-i\theta}T^{\sharp_A}\big)^{\sharp_A}}{2} \big(S^{\sharp_A}\big)^{\sharp_A}
+ S^{\sharp_A} \frac{e^{-i\theta}T^{\sharp_A} + \big(e^{-i\theta}T^{\sharp_A}\big)^{\sharp_A}}{2}.
\end{align}
Therefore, by Theorem \ref{T.2.30} and (\ref{I.1.L.2.30}), we obtain
\begin{align*}
w_A\Big(\big((TS)^{\sharp_A} &+ T^{\sharp_A} S\big)^{\sharp_A}\Big)
\\& = \displaystyle{\sup_{\theta \in \mathbb{R}}}{\left\|\frac{e^{i\theta}\big(((TS)^{\sharp_A} + T^{\sharp_A} S)\big)^{\sharp_A} + \Big(e^{i\theta}\big(((TS)^{\sharp_A} + T^{\sharp_A} S)\big)^{\sharp_A}\Big)^{\sharp_A}}{2}\right\|}_A
\\&\leq \displaystyle{\sup_{\theta \in \mathbb{R}}}{\left\|\frac{e^{-i\theta} T^{\sharp_A} + \big(e^{-i\theta}T^{\sharp_A}\big)^{\sharp_A}}{2} \big(S^{\sharp_A}\big)^{\sharp_A}
+ S^{\sharp_A} \frac{e^{-i\theta}T^{\sharp_A} + \big(e^{-i\theta}T^{\sharp_A}\big)^{\sharp_A}}{2}\right\|}_A
\\&\leq \displaystyle{\sup_{\theta \in \mathbb{R}}}{\left\|\frac{e^{-i\theta} T^{\sharp_A} + \big(e^{-i\theta}T^{\sharp_A}\big)^{\sharp_A}}{2}\right\|}_A
\Big({\|\big(S^{\sharp_A}\big)^{\sharp_A}\|}_A
+ {\|S^{\sharp_A} \|}_A \Big)
\\& = w_A(T^{\sharp_A})({\|S\|}_A + {\|S\|}_A)
\\& \qquad \qquad \qquad\big(\mbox{since ${\|R^{\sharp_A}\|}_A = {\|R\|}_A$ for every $R\in\mathbb{B}_{A}(\mathcal{H})$}\big)
\\&= 2w_A(T){\|S\|}_A.
\end{align*}
Hence
\begin{align}\label{I.01.L.2.10}
w_A\Big(\big((TS)^{\sharp_A} &+ T^{\sharp_A} S\big)^{\sharp_A}\Big) \leq 2w_A(T){\|S\|}_A.
\end{align}
Since $w_A(R^{\sharp_A}) = w_A(R)$ for every $R\in\mathbb{B}_{A}(\mathcal{H})$, from (\ref{I.01.L.2.10}) we obtain
\begin{align}\label{I.1.L.2.10}
w_A\Big((TS)^{\sharp_A} + T^{\sharp_A} S\Big) \leq 2w_A(T){\|S\|}_A.
\end{align}
Finally, by replacing $S$ by $-iS$ in (\ref{I.1.L.2.10}), we reach that
\begin{align*}
w_A\Big((TS)^{\sharp_A} - T^{\sharp_A} S\Big)\leq 2w_A(T){\|S\|}_A .
\end{align*}
\end{proof}
In the next theorem, we give a new upper bound
for the $A$-numerical radius of products of semi-Hilbertian space operators.
\begin{theorem}\label{T.2.8}
Let $T, S\in\mathbb{B}_{A}(\mathcal{H})$. Then
\begin{align*}
w_A(TS) \leq w_A(T){\|S\|}_A + \frac{1}{2}w_A\Big((TS)^{\sharp_A} \pm T^{\sharp_A} S\Big) \leq 2w_A(T){\|S\|}_A.
\end{align*}
\end{theorem}
\begin{proof}
The second inequality follows from Lemma \ref{L.2.10}.
It is therefore enough to prove the first inequality.
Let $\theta \in \mathbb{R}$. By Lemma \ref{l.2.1.5} we have
\begin{align*}
&{\left\|\frac{e^{i\theta}TS + (e^{i\theta}TS)^{\sharp_A}}{2}\right\|}_A
= w_A\left(\frac{e^{i\theta}TS + (e^{i\theta}TS)^{\sharp_A}}{2}\right)
\\& \qquad = w_A\left(\frac{e^{i\theta}T + (e^{i\theta}T)^{\sharp_A}}{2} S
+ e^{-i\theta}\frac{S^{\sharp_A} T^{\sharp_A} - T^{\sharp_A} S}{2}\right)
\\& \qquad \leq w_A\left(\frac{e^{i\theta}T + (e^{i\theta}T)^{\sharp_A}}{2} S\right)
+ w_A\left(e^{-i\theta}\frac{S^{\sharp_A} T^{\sharp_A} - T^{\sharp_A} S}{2}\right)
\\& \qquad \leq {\left\|\frac{e^{i\theta}T + (e^{i\theta}T)^{\sharp_A}}{2} S\right\|}_A
+ \frac{1}{2}w_A\Big(S^{\sharp_A} T^{\sharp_A} - T^{\sharp_A} S\Big)
\qquad\big(\mbox{by Corollary \ref{C.2.6}}\big)
\\& \qquad \leq {\left\|\frac{e^{i\theta}T + (e^{i\theta}T)^{\sharp_A}}{2}\right\|}_A {\|S\|}_A
+ \frac{1}{2}w_A\Big((TS)^{\sharp_A} - T^{\sharp_A} S\Big)
\\& \qquad \leq w_A(T){\|S\|}_A + \frac{1}{2}w_A\Big((TS)^{\sharp_A} - T^{\sharp_A} S\Big)
\qquad \qquad\big(\mbox{by Theorem \ref{T.2.30}}\big).
\end{align*}
Thus
\begin{align}\label{I.1.T.2.8}
w_A(TS) = \displaystyle{\sup_{\theta \in \mathbb{R}}}{\left\|\frac{e^{i\theta}TS + (e^{i\theta}TS)^{\sharp_A}}{2}\right\|}_A
\leq w_A(T){\|S\|}_A + \frac{1}{2}w_A\Big((TS)^{\sharp_A} - T^{\sharp_A} S\Big).
\end{align}
Now, by replacing $S$ by $-iS$ in (\ref{I.1.T.2.8}), we conclude that
\begin{align}\label{I.2.T.2.8}
w_A(TS) \leq w_A(T){\|S\|}_A + \frac{1}{2}w_A\Big((TS)^{\sharp_A} + T^{\sharp_A} S\Big).
\end{align}
By (\ref{I.1.T.2.8}) and (\ref{I.2.T.2.8}) we deduce the desired result.
\end{proof}
As an immediate consequence of the preceding theorem, we have the following result.
\begin{corollary}\label{C.2.9}
Let $T, S\in\mathbb{B}_{A}(\mathcal{H})$.
If $(TS)^{\sharp_A} = T^{\sharp_A} S$, then
\begin{align*}
w_A(TS) \leq w_A(T){\|S\|}_A.
\end{align*}
\end{corollary}
In the following, for $R\in\mathbb{B}_{A}(\mathcal{H})$, let $d_A(R)$ denote the $A$-numerical radius distance
of $R$ from the scalar operators, that is,
\begin{align*}
d_A(R) = \inf\big\{w_A(R + \zeta I): \,\, \zeta \in \mathbb{C}\big\}.
\end{align*}
Next, we present a improvement of the second inequality in (\ref{I.10.1}).
\begin{theorem}\label{T.2.11}
Let $T, S\in\mathbb{B}_{A}(\mathcal{H})$. Then
\begin{align*}
w_A(TS ) &\leq {\|TS\|}_A
\\& \leq \min\Big\{{\|T\|}_A\big(w_A(S) + d_A(S)\big), {\|S\|}_A\big(w_A(T) + d_A(T)\big)\Big\}
\\& \leq 2\min\Big\{{\|T\|}_A w_A(S), {\|S\|}_A w_A(T)\Big\}.
\end{align*}
\end{theorem}
\begin{proof}
Using a compactness argument, let $\zeta_0 \in \mathbb{C}$ such that $d_A(T) = w_A(T + \zeta_0I)$.
If $\zeta_0 = 0$, then by the inequalities of (\ref{I.1.C.2.6}) we get
\begin{align}\label{I.1.T.2.11}
w_A(TS ) \leq {\|TS\|}_A \leq 2{\|S\|}_A w_A(T) = {\|S\|}_A\big(w_A(T) + d_A(T)\big).
\end{align}
Hence, we may assume that $\zeta_0 \neq 0$. Put $\zeta = \frac{\overline{\zeta_0}}{|\zeta_0|}$. Then
\begin{align*}
w_A(TS) &\leq {\|TS\|}_A \leq {\|T\|}_A{\|S\|}_A
\\& = {\left\|\frac{(\zeta T) + (\zeta T)^{\sharp_A}}{2} + i\frac{(\zeta T) - (\zeta T)^{\sharp_A}}{2i}\right\|}_A{\|S\|}_A
\\& \leq \left({\left\|\frac{(\zeta T) + (\zeta T)^{\sharp_A}}{2}\right\|}_A
+ {\left\|\frac{(\zeta T) - (\zeta T)^{\sharp_A}}{2i}\right\|}_A\right){\|S\|}_A
\\& = \left({\left\|\frac{(\zeta T) + (\zeta T)^{\sharp_A}}{2}\right\|}_A
+ {\left\|\frac{\big(\zeta (T + \zeta_0I)\big) - \big(\zeta (T + \zeta_0I)\big)^{\sharp_A}}{2i}\right\|}_A\right){\|S\|}_A
\\& \leq \Big(w_A(\zeta T) + w_A\big(\zeta (T + \zeta_0I)\big)\Big){\|S\|}_A \qquad \qquad\big(\mbox{by Corollary \ref{C.2.5}}\big)
\\& = \big(w_A(T) + d_A(T)\big){\|S\|}_A
\end{align*}
Hence
\begin{align}\label{I.2.T.2.11}
w_A(TS) \leq {\|TS\|}_A \leq \big(w_A(T) + d_A(T)\big) {\|S\|}_A.
\end{align}
Since $d_A(T) \leq w_A(T)$, from (\ref{I.1.T.2.11}) and (\ref{I.2.T.2.11}) it follows that
\begin{align}\label{I.3.T.2.11}
w_A(TS )\leq {\|TS\|}_A \leq {\|S\|}_A\big(w_A(T) + d_A(T)\big) \leq 2w_A(T){\|S\|}_A.
\end{align}
By a similar argument we have
\begin{align}\label{I.4.T.2.11}
w_A(TS )\leq {\|TS\|}_A \leq {\|T\|}_A\big(w_A(S) + d_A(S)\big) \leq 2w_A(S){\|T\|}_A.
\end{align}
Now, by (\ref{I.3.T.2.11}) and (\ref{I.4.T.2.11}) we obtain the desired inequalities.
\end{proof}
We finish this section by another upper bound
for the $A$-numerical radius of products of semi-Hilbertian space operators.
\begin{theorem}\label{T.2.12}
Let $T, S\in\mathbb{B}_{A}(\mathcal{H})$. Then
\begin{align*}
w_A(TS ) \leq {\|TS\|}_A \leq \big(w_A(T) + d_A(T)\big)\big(w_A(S) + d_A(S)\big)\leq 4w_A(T)w_A(S).
\end{align*}
\end{theorem}
\begin{proof}
The fact that $d_A(R) \leq w_A(R)$ holds for every $R\in\mathbb{B}_{A}(\mathcal{H})$ implies that
the third desired inequality.

Now, let $\zeta_0, \xi_0 \in \mathbb{C}$ such that $d_A(T) = w_A(T + \zeta_0I)$ and $d_A(S) = w_A(S + \xi_0I)$.
As in the proof of Theorem \ref{T.2.11} we may assume that $\zeta_0 \xi_0 \neq 0$.
Put $\zeta = \frac{\overline{\zeta_0}}{|\zeta_0|}$ and $\xi = \frac{\overline{\xi_0}}{|\xi_0|}$.
Therefore, we have
\begin{align*}
w_A(TS) &\leq {\|TS\|}_A
\\&\leq {\|T\|}_A{\|S\|}_A
\\& = {\left\|\frac{(\zeta T) + (\zeta T)^{\sharp_A}}{2} + i\frac{(\zeta T) - (\zeta T)^{\sharp_A}}{2i}\right\|}_A
\\& \qquad \qquad \qquad\times{\left\|\frac{(\xi S) + (\xi S)^{\sharp_A}}{2} + i\frac{(\xi S) - (\xi S)^{\sharp_A}}{2i}\right\|}_A
\\& \leq \left({\left\|\frac{(\zeta T) + (\zeta T)^{\sharp_A}}{2}\right\|}_A + {\left\|\frac{(\zeta T) - (\zeta T)^{\sharp_A}}{2i}\right\|}_A\right)
\\& \qquad \qquad \qquad\times\left({\left\|\frac{(\xi S) + (\xi S)^{\sharp_A}}{2}\right\|}_A + {\left\|\frac{(\xi S) - (\xi S)^{\sharp_A}}{2i}\right\|}_A\right)
\\& = \left({\left\|\frac{(\zeta T) + (\zeta T)^{\sharp_A}}{2}\right\|}_A
+ {\left\|\frac{\big(\zeta (T + \zeta_0I)\big) - \big(\zeta (T + \zeta_0I)\big)^{\sharp_A}}{2i}\right\|}_A\right)
\\& \qquad \qquad\times\left({\left\|\frac{(\xi S) + (\xi S)^{\sharp_A}}{2}\right\|}_A
+ {\left\|\frac{\big(\xi (S + \xi_0I)\big) - \big(\xi (S + \xi_0I)\big)^{\sharp_A}}{2i}\right\|}_A\right)
\\& \leq \Big(w_A(\zeta T) + w_A\big(\zeta (T + \zeta_0I)\big)\Big)\Big(w_A(\xi S) + w_A\big(\xi (S + \xi_0I)\big)\Big)
\\& \qquad \qquad \qquad \qquad \qquad \qquad \qquad \qquad \qquad \qquad \qquad\big(\mbox{by Corollary \ref{C.2.5}}\big)
\\& = \big(w_A(T) + d_A(T)\big)\big(w_A(S) + d_A(S)\big).
\end{align*}
\end{proof}
\section{Upper bounds for the $A$-numerical radius of commutators, and anticommutators of operators}
In this section, we present some upper bounds for the $A$-numerical radius of commutators,
and anticommutators of semi-Hilbertian space operators.
To achieve the first main result in this section, we need the following lemma.
\begin{lemma}\label{l.2.13.9}
Let $R\in\mathbb{B}_{A}(\mathcal{H})$. Then
\begin{align*}
{\big\|R^{\sharp_A} R + RR^{\sharp_A}\big\|}_A \leq 2\big(w^2_A(R) +d^2_A(R)\big)\leq 4w^2_A(R).
\end{align*}
\end{lemma}
\begin{proof}
Observe that, from $d_A(R)\leq w_A(R)$ we have
$2\big(w^2_A(R) +d^2_A(R)\big)\leq 4w^2_A(R)$.
It is therefore enough to prove the first inequality.
Let $\zeta_0 \in \mathbb{C}$ such that $d_A(R) = w_A(R + \zeta_0I)$.
If $\zeta_0 = 0$, then by employing Corollary \ref{C.2.5} we have
\begin{align*}
{\big\|R^{\sharp_A} R + RR^{\sharp_A}\big\|}_A &
= {\left\|2 \left(\frac{R + R^{\sharp_A}}{2}\right)^2 + 2 \left(\frac{R - R^{\sharp_A}}{2i}\right)^2\right\|}_A
\\& \leq 2{\left\|\frac{R + R^{\sharp_A}}{2}\right\|}^2_A + 2{\left\|\frac{R - R^{\sharp_A}}{2i}\right\|}^2_A
\\& \leq 2w^2_A(R) + 2w^2_A(R) = 2\big(w^2_A(R) +d^2_A(R)\big).
\end{align*}
If $\zeta_0 \neq 0$, then put $\zeta = \frac{\overline{\zeta_0}}{|\zeta_0|}$.
A simple computation together with Corollary \ref{C.2.5} gives
\begin{align*}
{\big\|R^{\sharp_A} R + RR^{\sharp_A}\big\|}_A &
= {\left\|2 \left(\frac{(\zeta R) + (\zeta R)^{\sharp_A}}{2}\right)^2
+ 2 \left(\frac{(\zeta R) - (\zeta R)^{\sharp_A}}{2i}\right)^2\right\|}_A
\\& \leq 2{\left\|\frac{(\zeta R) + (\zeta R)^{\sharp_A}}{2}\right\|}^2_A
+ 2{\left\|\frac{(\zeta R) - (\zeta R)^{\sharp_A}}{2i}\right\|}^2_A
\\& = 2{\left\|\frac{(\zeta R) + (\zeta R)^{\sharp_A}}{2}\right\|}^2_A
+ 2{\left\|\frac{\zeta (R + \zeta_0I) - (\zeta (R + \zeta_0I))^{\sharp_A}}{2i}\right\|}^2_A
\\& \leq 2w^2_A\big(\zeta R\big) + 2w^2_A\big(\zeta (R + \zeta_0I)\big)
\\& = 2\big(w^2_A(R) +d^2_A(R)\big).
\end{align*}
\end{proof}
The following result may be stated as well.
\begin{theorem}\label{T.2.15}
Let $T, S\in\mathbb{B}_{A}(\mathcal{H})$. Then
\begin{align*}
w_A(TS \pm ST) &\leq \sqrt{{\big\|TT^{\sharp_A} + T^{\sharp_A} T\big\|}_A}\sqrt{{\big\|SS^{\sharp_A} + S^{\sharp_A} S\big\|}_A}
\\& \leq 2\min\Big\{{\|T\|}_A\sqrt{w^2_A(S) + d^2_A(S)}, {\|S\|}_A\sqrt{w^2_A(T) + d^2_A(T)}\Big\}
\\& \leq 2\sqrt{2}\min\Big\{{\|T\|}_A w_A(S), {\|S\|}_A w_A(T)\Big\}.
\end{align*}
\end{theorem}
\begin{proof}
Clearly,
\begin{align*}
\min\Big\{{\|T\|}_A\sqrt{w^2_A(S) + d^2_A(S)}&, {\|S\|}_A\sqrt{w^2_A(T) + d^2_A(T)}\Big\}
\\& \leq \sqrt{2}\min\Big\{{\|T\|}_A w_A(S), {\|S\|}_A w_A(T)\Big\}.
\end{align*}
Now, let $x\in \mathcal{H}$ with ${\|x\|}_A = 1$. By the Cauchy--Schwarz inequality, we have
\begin{align*}
\Big|{\big\langle (TS \pm ST) x, x \big\rangle}_A\Big|^2
&\leq \Big(\big|{\langle TS x, x\rangle}_A\big| + \big|{\langle ST x, x\rangle}_A\big|\Big)^2
\\& = \Big(\big|{\langle S x, T^{\sharp_A} x\rangle}_A\big| + \big|{\langle T x, S^{\sharp_A} x\rangle}_A\big|\Big)^2
\\& \leq \Big({\|Sx\|}_A {\|T^{\sharp_A} x\|}_A + {\|Tx\|}_A{\|S^{\sharp_A} x\|}_A\Big)^2
\\& \leq \Big({\|Tx\|}^2_A  + {\|T^{\sharp_A} x\|}^2_A\Big)\Big({\|Sx\|}^2_A + {\|S^{\sharp_A} x\|}^2_A\Big)
\\& = {\Big\langle x, \big(T^{\sharp_A}T + TT^{\sharp_A}\big)x\Big\rangle}_A
{\Big\langle x, \big(S^{\sharp_A}S + SS^{\sharp_A}\big) x\Big\rangle}_A
\\& \leq {\big\|T^{\sharp_A}T + TT^{\sharp_A}\big\|}_A{\big\|S^{\sharp_A}S + SS^{\sharp_A}\big\|}_A.
\end{align*}
Thus
\begin{align*}
\Big|{\big\langle (TS \pm ST) x, x \big\rangle}_A\Big|
\leq \sqrt{{\big\|TT^{\sharp_A} + T^{\sharp_A} T\big\|}_A}\sqrt{{\big\|SS^{\sharp_A} + S^{\sharp_A} S\big\|}_A}.
\end{align*}
Taking the supremum over $x\in \mathcal{H}$ with ${\|x\|}_A = 1$ in the above inequality we get
\begin{align}\label{I.0.T.2.15}
w_A(TS \pm ST) \leq \sqrt{{\big\|TT^{\sharp_A} + T^{\sharp_A} T\big\|}_A}\sqrt{{\big\|SS^{\sharp_A} + S^{\sharp_A} S\big\|}_A}.
\end{align}
From (\ref{I.0.T.2.15}) and Lemma \ref{l.2.13.9} it follows that
\begin{align*}
w_A(TS \pm ST) &\leq \sqrt{{\big\|TT^{\sharp_A} + T^{\sharp_A} T\big\|}_A}\sqrt{{\big\|SS^{\sharp_A} + S^{\sharp_A} S\big\|}_A}
\\& \leq \sqrt{{\big\|TT^{\sharp_A}\big\|}_A + {\big\|T^{\sharp_A} T\big\|}_A}\sqrt{2\big(w^2_A(S) +d^2_A(S)\big)}
\\& \leq 2{\|T\|}_A\sqrt{w^2_A(S) + d^2_A(S)},
\end{align*}
whence
\begin{align}\label{I.1.T.2.15}
w_A(TS \pm ST) &\leq \sqrt{{\big\|TT^{\sharp_A} + T^{\sharp_A} T\big\|}_A}\sqrt{{\big\|SS^{\sharp_A} + S^{\sharp_A} S\big\|}_A}
\nonumber
\\& \leq 2{\|T\|}_A\sqrt{w^2_A(S) + d^2_A(S)}.
\end{align}
Similarly,
\begin{align}\label{I.2.T.2.15}
w_A(TS \pm ST) &\leq \sqrt{{\big\|SS^{\sharp_A} + S^{\sharp_A} S\big\|}_A}\sqrt{{\big\|TT^{\sharp_A} + T^{\sharp_A} T\big\|}_A}
\nonumber
\\& \leq 2{\|S\|}_A\sqrt{w^2_A(T) + d^2_A(T)}.
\end{align}
Hence by (\ref{I.1.T.2.15}) and (\ref{I.2.T.2.15}) we deduce the desired result.
\end{proof}
As a consequence of Lemma \ref{l.2.13.9} and Theorem \ref{T.2.15}, we have the following result.
\begin{corollary}\label{C.2.17}
Let $T, S\in\mathbb{B}_{A}(\mathcal{H})$. Then
\begin{align*}
w_A(TS \pm ST) \leq 2\sqrt{w^2_A(T) + d^2_A(T)}\sqrt{w^2_A(S) + d^2_A(S)} \leq 4w_A(T)w_A(S).
\end{align*}
\end{corollary}
For the second main result in this section, we need the following lemma
that is interesting on its own right.
\begin{lemma}\label{l.2.13}
For $T, S, R\in\mathbb{B}_{A}(\mathcal{H})$ the following statements hold.
\begin{itemize}
\item[(i)] $w_A\Big(TRT^{\sharp_A}\Big) \leq {\|T\|}^2_A w_A(R)$.
\item[(ii)] $w_A\Big(SRT^{\sharp_A}\Big) \leq \frac{1}{2} {\big\|TT^{\sharp_A} + SS^{\sharp_A}\big\|}_A {\|R\|}_A$.
\end{itemize}
\end{lemma}
\begin{proof}
(i) Let $x\in \mathcal{H}$ with ${\|x\|}_A = 1$. We have
\begin{align*}
\Big|{\big\langle TRT^{\sharp_A} x, x \big\rangle}_A\Big| = \Big|{\big\langle RT^{\sharp_A} x, T^{\sharp_A} x \big\rangle}_A\Big|
\leq w_A(R){\|T^{\sharp_A} x\|}^2_A \leq w_A(R){\|T\|}^2_A.
\end{align*}
Now, by taking the supremum over all $x\in \mathcal{H}$ with ${\|x\|}_A = 1$ we conclude that
\begin{align*}
w_A\Big(TRT^{\sharp_A}\Big) \leq {\|T\|}^2_A w_A(R)
\end{align*}
(ii) Let $x\in \mathcal{H}$ with ${\|x\|}_A = 1$. We have
\begin{align*}
\Big|{\big\langle SRT^{\sharp_A} x, x \big\rangle}_A\Big| &= \Big|{\big\langle T^{\sharp_A} x, R^{\sharp_A} S^{\sharp_A} x \big\rangle}_A\Big|
\\& \leq {\|T^{\sharp_A} x\|}_A{\|R^{\sharp_A} S^{\sharp_A} x\|}_A
\\& \leq {\|T^{\sharp_A} x\|}_A{\|S^{\sharp_A} x\|}_A {\|R^{\sharp_A}\|}_A
\\& \leq \frac{1}{2}\Big({\|T^{\sharp_A} x\|}^2_A + {\|S^{\sharp_A} x\|}^2_A\Big){\|R\|}_A
\\& \qquad \qquad \big(\mbox{by the arithmetic geometric mean inequality}\big)
\\& = \frac{1}{2}\Big({\big\langle x, TT^{\sharp_A} x \big\rangle}_A + {\big\langle x, SS^{\sharp_A} x \big\rangle}_A\Big){\|R\|}_A
\\& = \frac{1}{2}{\Big\langle x, \big(TT^{\sharp_A} + SS^{\sharp_A}\big) x \big\rangle}_A{\|R\|}_A
\\& \leq \frac{1}{2} {\big\|TT^{\sharp_A} + SS^{\sharp_A}\big\|}_A {\|R\|}_A,
\end{align*}
which, by taking the supremum over $x\in \mathcal{H}$, ${\|x\|}_A = 1$, implies that
\begin{align*}
w_A\Big(SRT^{\sharp_A}\Big) \leq \frac{1}{2} {\big\|TT^{\sharp_A} + SS^{\sharp_A}\big\|}_A {\|R\|}_A.
\end{align*}
\end{proof}
Finally, we present the following result.
\begin{theorem}\label{T.2.14}
Let $T, S\in\mathbb{B}_{A}(\mathcal{H})$. Then
\begin{align*}
w_A(TS^{\sharp_A} \pm ST^{\sharp_A}) \leq {\big\|T^{\sharp_A} T + SS^{\sharp_A}\big\|}_A.
\end{align*}
\end{theorem}
\begin{proof}
Let $\theta \in \mathbb{R}$. We have
\begin{align*}
&{\left\|\frac{e^{i\theta}(TS^{\sharp_A} + ST^{\sharp_A}) + \big(e^{i\theta}(TS^{\sharp_A} + ST^{\sharp_A})\big)^{\sharp_A}}{2}\right\|}_A
\\& \qquad \qquad = {\left\|\frac{\Big(e^{i\theta}(TS^{\sharp_A} + ST^{\sharp_A})\Big)^{\sharp_A} + \Big(\big(e^{i\theta}(TS^{\sharp_A} + ST^{\sharp_A})\big)^{\sharp_A}\Big)^{\sharp_A}}{2}\right\|}_A
\\& \qquad \qquad \qquad \qquad \qquad \qquad\big(\mbox{since ${\|R^{\sharp_A}\|}_A = {\|R\|}_A$ for every $R\in\mathbb{B}_{A}(\mathcal{H})$}\big)
\\& \qquad \qquad = {\left\|\frac{(T^{\sharp_A})^{\sharp_A}(e^{i\theta} I + e^{-i\theta}I)S^{\sharp_A}  + \Big((T^{\sharp_A})^{\sharp_A}(e^{i\theta} I + e^{-i\theta}I)S^{\sharp_A}\Big)^{\sharp_A}}{2}\right\|}_A
\\& \qquad \qquad \qquad \qquad \qquad \qquad\big(\mbox{since $\big((R^{\sharp_A})^{\sharp_A}\big)^{\sharp_A} = R^{\sharp_A}$ for every $R\in\mathbb{B}_{A}(\mathcal{H})$}\big)
\\& \qquad \qquad \leq w_A\Big((T^{\sharp_A})^{\sharp_A}(e^{i\theta} I + e^{-i\theta}I)S^{\sharp_A}\Big)
\qquad \qquad \qquad\big(\mbox{by Lemma \ref{l.2.1.5}}\big)
\\& \qquad \qquad = w_A\Big(S(e^{-i\theta} I + e^{i\theta}I)T^{\sharp_A}\Big)
\\& \qquad \qquad \qquad \qquad \qquad \qquad\big(\mbox{since $w_A(R^{\sharp_A}) = w_A(R)$ for every $R\in\mathbb{B}_{A}(\mathcal{H})$}\big)
\\& \qquad \qquad \leq \frac{1}{2} {\big\|TT^{\sharp_A} + SS^{\sharp_A}\big\|}_A {\big\|e^{-i\theta} I + e^{i\theta}I\big\|}_A
\qquad \qquad\big(\mbox{by Lemma \ref{l.2.13} (ii)}\big)
\\& \qquad \qquad = {\big\|TT^{\sharp_A} + SS^{\sharp_A}\big\|}_A.
\end{align*}
Thus
\begin{align*}
{\left\|\frac{e^{i\theta}(TS^{\sharp_A} + ST^{\sharp_A}) + \big(e^{i\theta}(TS^{\sharp_A} + ST^{\sharp_A})\big)^{\sharp_A}}{2}\right\|}_A
\leq {\big\|TT^{\sharp_A} + SS^{\sharp_A}\big\|}_A,
\end{align*}
and so,
\begin{align*}
\displaystyle{\sup_{\theta \in \mathbb{R}}}
{\left\|\frac{e^{i\theta}(TS^{\sharp_A} + ST^{\sharp_A}) + \big(e^{i\theta}(TS^{\sharp_A} + ST^{\sharp_A})\big)^{\sharp_A}}{2}\right\|}_A
\leq {\big\|TT^{\sharp_A} + SS^{\sharp_A}\big\|}_A.
\end{align*}
Then, by Theorem \ref{T.2.30}, we get
\begin{align}\label{I.1.T.2.14}
w_A(TS^{\sharp_A} + ST^{\sharp_A}) \leq {\big\|T^{\sharp_A} T + SS^{\sharp_A}\big\|}_A.
\end{align}
Finally, by replacing $T$ by $iT$ in (\ref{I.1.T.2.14}), we obtain
\begin{align*}
w_A(TS^{\sharp_A} - ST^{\sharp_A}) \leq {\big\|T^{\sharp_A} T + SS^{\sharp_A}\big\|}_A,
\end{align*}
and the proof is completed.
\end{proof}

{\bf Acknowledgments.} 
The author expresses his gratitude to the referee for his/hers comments towards an improved final
version of the paper. He would also like to thank Professor M. S. Moslehian for his helpful suggestions.
This work was supported by a grant from Shanghai Municipal Science and Technology Commission (18590745200).
\bibliographystyle{amsplain}

\end{document}